\documentclass[runningheads]{llncs}
\usepackage[T1]{fontenc}
\usepackage{graphicx}
\usepackage{amsmath,amssymb}
\usepackage[
draft=false]{hyperref}
\usepackage[capitalize]{cleveref}
\usepackage{color}

\begin{document}
\title{Bi-forms Approach to Potential Functions in Information Geometry}
%
%
\author{Florio M. Ciaglia\inst{1}
\and
Giuseppe Marmo\inst{2,3}
\and
Marco Pacelli\inst{3,4}\thanks{Corresponding author: marco.pacelli-ssm@unina.it}
\and 
Luca Schiavone\inst{3,5}
\and 
Alessandro Zampini\inst{3,4,5}
}
\authorrunning{F. M. Ciaglia et al.}
%
\institute{Department of Mathematics, University Carlos III de Madrid, Leganés, Madrid, Spain \\ \email{ fciaglia@math.uc3m.es}
\\ \and
Dipartimento di Fisica “E. Pancini”, Università di Napoli Federico II, Naples, Italy\\ \and 
INFN-Sezione di Napoli, Naples, Italy\\
\email{giuseppe.marmo@na.infn.it}
\\ \and 
Scuola Superiore Meridionale, Naples, Italy\\
\email{marco.pacelli-ssm@unina.it}
\and
Dipartimento di Matematica e Applicazioni “Renato Caccioppoli”, Università di Napoli Federico II, Naples, Italy\\
\email{luca.schiavone@unina.it,alessandro.zampini@unina.it}
}
\maketitle              
\begin{abstract}
Contrast functions play a fundamental role in information geometry, providing a means for generating the geometric structures of a statistical manifold: a pseudo-Riemannian metric and a pair of torsion-free conjugate affine connections. Conventional contrast-based approaches become indeed insufficient within settings where torsion is naturally present, such as quantum information geometry. This paper introduces contrast bi-forms, a generalisation of contrast functions that systematically encode metric and connection data, allowing for arbitrary affine connections regardless of torsion. It will be shown that they provide a unified framework for statistical potentials, offering new insights into the inverse problem in information geometry. As an example, we consider teleparallel manifolds, where torsion is intrinsic to the geometry, and show how bi-forms naturally accommodate these structures.

\end{abstract}
\section{Introduction}

Contrast (or divergence) functions on a smooth manifold play a fundamental role in both mathematical theory and applications, providing a measure of distinguishability between probability distributions \cite{ref_Egu}. A crucial feature of contrast functions is that they naturally induce a pseudo-Riemannian metric $g$ and a pair of $g$-conjugate and torsion-free affine connections $(\nabla,\nabla^\dag)$ on the underlying manifold $M$ \cite{ref_Egu}. The resulting structure defines a statistical manifold in the sense of Lauritzen, capturing the essential mathematical features of parametric statistical models \cite{ref_Lau}. Within such paradigmatic examples, the metric $g$ coincides with the Fisher-Rao metric, while the torsion-free affine connections $\nabla$ and $\nabla^\dag$ correspond to the $(\pm 1)$-connections. 
While this framework is well suited to many applications, there are cases where the torsion-free condition does not naturally hold. A key example arises from quantum finite level systems, where the geometry of the space of faithful quantum states suggests that torsion is an inherent feature rather than an anomaly \cite{ref_CdC,ref_J}. This motivates the introduction of a broader class of statistical structures. A statistical manifold admitting torsion (or, simply, SMAT) is a triple $(M, g, \nabla)$, where $\nabla$ is required to be torsion-free, whereas $\nabla^\dag$ may exhibit torsion. Such structures have been explored in connection with estimation theory \cite{ref_HM_2017} and in attempts to formulate a geometric framework for quantum information theory \cite{ref_Kurose}. 
Extending the idea of SMATs, in this paper we introduce what we call Lauritzen manifolds, that is, triples $(M, g, \nabla)$ in which both $\nabla$ and $\nabla^\dag$ may possess torsion. A natural question then arises: can these structures still be derived from a potential? As anticipated above, contrast functions are not suitable candidates, since Schwarz's theorem forces any connection derived from a contrast function to be torsion-free. This limitation has been the starting point to explore alternative formulations for potentials. We limit ourselves here to recall that  Henmi and Matsuzoe \cite{ref_HM} introduced pre-contrast functions to generate SMATs, while Zhang and Khan \cite{ref_ZK} proposed super-contrast functions as potentials for Lauritzen manifolds, both extending the standard theory of contrast functions in order to accommodate torsion.
In this paper, we propose an alternative method for generating Lauritzen manifolds. Instead of relying on super-contrast functions, we adopt the formalism of bi-forms \cite{ref_EB_I,ref_EB_II,ref_NSS,ref_Ruse} and introduce the notion of a contrast bi-form as a suitable potential. This approach, in our opinion, offers two key advantages. First, it provides a unified framework for describing potentials, since contrast, pre-contrast, and super-contrast functions come as specific instances of bi-forms. Second, it offers a natural framework to address what we call the inverse problem in information geometry: given a statistical or Lauritzen manifold, determining whether a potential can generate it. While this problem is well understood in the case of statistical manifolds \cite{ref_AA,ref_Matu}, it remains largely unexplored for SMATs and Lauritzen manifolds. Upon using the general framework of bi-forms, and developping from \cite{ref_HM,ref_ZK}, we aim to establish a solid foundation for a systematic theory of potentials in information geometry. This provides a unified approach for  both classical and quantum settings.
Due to the space limitations of the present issue, we will describe the details and present full proofs of the statements and results discussed here in a forthcoming publication.
\section{Contrast bi-forms}  
The theory of contrast functions is based on the geometry of the Cartesian square of a manifold. Given a manifold $M$, its Cartesian square is the product manifold $M \times M$ equipped with the canonical projections $\pi_L\colon M\times M\to M$ and $\pi_R\colon M\times M\to M$, which extract the left and right components, respectively:
\begin{align}
    \pi_L(m,n)&=m \qquad \textup{ and } \qquad \pi_R(m,n)=n\,.
\end{align}
This geometric framework was first used by Eguchi \cite{ref_Egu} to geometrically describe an algorithm that associates a pseudo-Riemannian metric $g$ and a pair of $g$-conjugate, torsion-free affine connections to a suitable class of \textit{2-point functions} on $M$, i.e., smooth real-valued functions defined on the product manifold $M \times M$ known as contrast functions. In the terminology commonly used in information geometry, these functions are also referred to as divergence functions or potential functions. While these terms may have slightly different technical meanings depending on the context or the class of functions considered, they all share the property of inducing the geometric structures $(g,\nabla)$ on $M$.

In this section, we introduce a more general formulation of these geometric objects using the formalism of bi-forms.

\subsection{Bi-forms}  

A bi-form on a manifold $M$ is a section of a tensor bundle over its Cartesian square $M \times M$. More precisely, given $(p,q) \in \mathbb{N}_0 \times \mathbb{N}_0$, a $(p,q)$-bi-form on $M$ is a smooth section of the vector bundle:  
\begin{equation}
\bigwedge^p\pi_L^\ast T^\vee M\otimes_{M\times M} \bigwedge^q\pi_R^\ast T^\vee M\,,
\end{equation} 
where $\tau_M^\vee\colon T^\vee M\to M$ is the cotangent bundle projection of $M$. The space of $(p,q)$-bi-forms is denoted by $\Omega^{p,q}(M\mid M)$.  

 A $(p,q)$-bi-form $\varpi$ can be interpreted as a $(p+q)$-multilinear real-valued fiber-wise function on $(TM)^{p+q}$ that is skew-symmetric in the first $p$ and in the last $q$ entries. This leads to a pairing between bi-forms on $M$ and vector fields on $M$. Given a $(p,q)$-bi-form $\varpi$ and vector fields $\{X_i\}_{i=1}^{p}$ and $\{Y_j\}_{j=1}^q$ on $M$, we define:
\begin{multline}
\varpi(X_1,\dots,X_p\mid Y_1,\dots,Y_q)(m,n) \\= \varpi(X_1(m),\dots,X_p(m)\mid Y_1(n),\dots,Y_q(n))\,.
\end{multline}
Furthermore, this pairing provides a systematic method for generating block-wise alternating covariant tensors on $M$, i.e., sections of the tensor bundle over $M$ whose total manifold is:\begin{equation}
\bigwedge^pT^\vee M\otimes_M \bigwedge^qT^\vee M\,.
\end{equation}
Both pseudo-Riemannian metrics and the torsion tensors of affine connections -- when interpreted as $3$-covariant tensors via the musical isomorphism -- are  of this type. These can be systematically obtained from bi-forms via the diagonal embedding $\iota\colon M\to M\times M$, defined by:
\begin{equation}
\iota(m)=(m,m)\,.
\end{equation} 
Indeed, this map induces the operator:
\begin{equation}
\iota^\ast \colon \Omega^{p,q}(M\mid M)\to \Omega^p(M)\otimes_{C^\infty(M)}\Omega^q(M)\,,
\end{equation} 
where the tensor product is taken over the ring of smooth functions on $M$. This operator assigns to any $(p,q)$-bi-form $\varpi$ on $M$ the block-wise alternating tensor:  
\begin{equation}\label{Eq: iota ast}
(\iota^\ast \varpi)(Z_1,\dots,Z_{p+q})=\iota^\ast\big( \varpi(Z_1,\dots,Z_p\mid Z_{p+1},\dots,Z_{p+q})\big)\,,
\end{equation} 
where $\{Z_i\}_{i=1}^{p+q}$ are vector fields on $M$. Since the pairing is local, the operator $\iota^\ast$ naturally restricts to the space of $(p,q)$-bi-forms that are defined on an open neighborhood of the diagonal submanifold $\Delta_M$ of $M \times M$. We denote this space by $\Omega^{p,q}_{\Delta_M}(M\mid M)$.

Another useful operator on bi-forms arises in terms of a distinguished map associated to the Cartesian square of a manifold, namely the swap map $s\colon M\times M\to M\times M$ defined by:
\begin{equation}
s(m,n) = (n,m)\,.
\end{equation}
This map interchanges the roles of the left and right projections, providing a systematic way to mirror properties between left and right  components. Since $s$ induces a fiber bundle isomorphism between $\pi_L$ and $\pi_R$, it induces the operator:
\begin{equation}
    s^\ast\colon \Omega^{p,q}(M\mid M)\to \Omega^{q,p}(M\mid M)\,,
\end{equation}
which acts on a $(p,q)$-bi-form $\varpi$ on $M$ as:
\begin{equation}
    (s^\ast\varpi)(X_1,\dots,X_q\mid Y_1,\dots,Y_p) = s^\ast\big(\varpi(Y_1,\dots,Y_p\mid X_1,\dots,X_q)\big)\,,
\end{equation}
where $\{X_i\}_{i=1}^q$ and $\{Y_j\}_{j=1}^p$ are vector fields on $M$.  
Since the pairing is local and $s$ fixes the diagonal submanifold $\Delta_M$ of $M\times M$, also the operator $s^\ast$ restricts to $s^\ast\colon \Omega^{p,q}_{\Delta_M}(M\mid M)\to \Omega^{q,p}_{\Delta_M}(M\mid M)$.

Finally, we discuss a structure on the space of bi-forms that has an interesting interpretation from the statistical point of view, that is the commutative bi-complex of bi-forms. The vertical distribution induced by  $\pi_R$ defines a flat generalized Ehresmann connection on the fiber bundle $\pi_L$, inducing a decomposition of the tangent bundle of $M \times M$ as the Whitney sum of the vertical distributions corresponding to $\pi_L$ and $\pi_R$. This decomposition naturally gives rise to a pair of differential operators $\mathrm{d}^L$ and $\mathrm{d}^R$ acting on bi-forms, that we define as follows.

The left differential operator $\mathrm{d}^L$ increases the left degree and is given by:
\begin{align}
    \mathrm{d}^L\colon \Omega^{p,q}(M\mid M)\to \Omega^{p+1,q}(M\mid M)\,,
\end{align}
with the explicit action:
\begin{multline}\label{Eq: global formula dl}
\mathrm{d}^{L}\varpi(X_0,X_1,\dots,X_p\mid Y_1,\dots,Y_q)\\
=\sum_{i=0}^p(-1)^i\mathcal{L}_{X_i^{L}}\big(\varpi(X_0,\dots, \Check{X}_i,\dots,X_p\mid Y_1,\dots,Y_q)\big)\\
\quad +\sum_{0\le a<b\le p}(-1)^{a+b}\varpi([X_a,X_b],X_0,\dots, \Check{X}_a,\dots, \Check{X}_b,\dots,X_p\mid Y_1,\dots,Y_q)\,,
\end{multline}
where $\varpi$ is a $(p,q)$-bi-form on $M$, the elements $\{X_i\}_{i=0}^p, \{Y_j\}_{j=1}^q$ are vector fields on $M$, and $X^L$ denotes the $\pi_L$-horizontal lift of the vector field $X$. The check notation indicates omission of the corresponding vector field. It is possible to prove that $\mathrm{d}^L$ satisfies the cohomology condition $\mathrm{d}^L\circ \mathrm{d}^L=0$,
which means that bi-forms can be studied within a differential complex.

We introduce the right differential operator $\mathrm{d}^R$, which increases the right degree, and is given by:
\begin{equation}
    \mathrm d^R\colon \Omega^{p,q}(M\mid M)\to \Omega^{p,q+1}(M\mid M)\,.
\end{equation}
Instead of writing its full expression explicitly, which would be analogous to \eqref{Eq: global formula dl}, we use the swap map $s^\ast$ to define it as:
\begin{equation}
    \mathrm{d}^R\varpi = s^\ast\big(\mathrm{d}^L(s^\ast\varpi)\big)\,,
\end{equation}
where $\varpi$ is a $(p,q)$-bi-form on $M$.
As for the case of the  $\mathrm{d}^L$ operator, the operator $\mathrm{d}^R$ satisfies the cohomology property $\mathrm{d}^R\circ \mathrm{d}^R=0$. Furthermore,  these two operators commute, i.e. $\mathrm{d}^L \circ \mathrm{d}^R = \mathrm{d}^R \circ \mathrm{d}^L$. Thus, the space of bi-forms is endowed with the structure of a commutative bi-complex.

\subsection{Contrast Bi-Forms}  

Contrast functions, as characterized by Eguchi, are $2$-point functions inducing pseudo-Riemannian metrics. Similarly, $2$-covariant tensors on a manifold $M$ arise from  $(2,0)$, $(1,1)$, $(0,2)$-bi-forms. Since pseudo-Riemannian metrics are symmetric, the only suitable choice for encoding such a structure is a $(1,1)$-bi-form.

A contrast bi-form is an element $\varpi\in \Omega_{\Delta_M}^{1,1}(M\mid M)$ such that $\iota^\ast \varpi$ defines a pseudo-Riemannian metric, denoted by $g^\varpi = \iota^\ast \varpi$ (cf. \eqref{Eq: iota ast}). The associated affine connection $\nabla^\varpi$ is defined implicitly by:
\begin{equation}\label{Eq: induced connection}
    g^\varpi\big(\nabla^\varpi_ZX,Y\big) = \iota^\ast\Big(\mathcal L_{Z^L}\big(\varpi(X\mid Y)\big)\Big)\,,
\end{equation}
where $X,Y,Z$ are vector fields on $M$ and $Z^L$ denotes the $\pi_L$-horizontal lift of $Z$. The dual affine connection is induced by $s^\ast \varpi$, and we denote it as $\varpi^\dag = s^\ast\varpi$.

\begin{remark}
    Viewing $\varpi$ as a fiber-wise bilinear function $TM\times TM\to \mathbb{R}$ is coherent with the definitions of super-contrast functions of Zhang and Khan \cite{ref_ZK}.
\end{remark}

The commutative bi-complex structure provides a characterization of torsion properties of the induced connections. 

\begin{proposition}
Let $\varpi$ be a contrast bi-form on a manifold $M$. The connection $\nabla^\varpi$ is torsion-free if and only if $\iota^\ast \mathrm{d}^L\varpi = 0$, while the dual connection $\nabla^{\varpi^\dag}$ is torsion-free if and only if $\iota^\ast \mathrm{d}^R\varpi = 0$.
\end{proposition}
\begin{proof}
Recall the definition of the left differential of $\varpi$ (cf. \eqref{Eq: global formula dl}):
\begin{equation*}
(\mathrm d^L\varpi)(X_0,X_1\mid Y)=\mathcal L_{X_0^L}\big(\varpi(X_1\mid Y)\big)-\mathcal L_{X_1^L}\big(\varpi(X_0\mid Y)\big)-\varpi([X_0,X_1]\mid Y)\,,
\end{equation*}
where $X_0,X_1$ and $Y$ are arbitrary vector fields on $M$. If we pull such an expression back along the diagonal map $\iota$, and apply  the definition of the induced metric and connection, we obtain:
\begin{equation}
\iota^\ast\big((\mathrm d^L\varpi)(X_0,X_1\mid Y)\big) = g^\varpi\big(\nabla^\varpi_{X_0}X_1 - \nabla^\varpi_{X_1}X_0 - [X_0,X_1], Y\big)\,.
\end{equation}
That is:
\begin{equation}
(\iota^\ast \mathrm d^L\varpi)(X_0,X_1,Y)=g^\varpi\big(\operatorname{Tor}^{\nabla^\varpi}(X_0,X_1),Y\big)\,.
\end{equation}
Since $g^\varpi$ is non-degenerate, the vanishing of $\iota^\ast \mathrm d^L\varpi$ is equivalent to the vanishing of the torsion tensor $\operatorname{Tor}^{\nabla^\varpi}$. The statement for $\nabla^{\varpi^\dag}$ follows similarly by considering $\mathrm d^R\varpi$.
\end{proof}

The bi-complex structure also characterizes SMAT and statistical manifolds, as the following results (whose proves we are forced, as already mentioned in the introduction, to omit) show. 

\begin{theorem}
    Let $\varpi$ be a contrast bi-form on a manifold $M$. If $\varpi$ is left-exact, i.e. there is $S\in \Omega^{0,1}(M\mid M)$ such that $\varpi=\mathrm d^LS$, then $(M,g^\varpi,\nabla^\varpi)$ is a SMAT. Conversely, if the contrast bi-form $\varpi$ induces a SMAT structure on $M$, then there is $S\in \Omega^{0,1}(M\mid M)$ such that $\varpi$ and $\mathrm d^LS$ produce the same pseudo-Riemannian metric and affine connection. Moreover, $S$ can be chosen in such a way that $\iota^\ast S=0$.
\end{theorem}
\begin{theorem}
    Let $\varpi$ be a contrast bi-form on a manifold $M$. If $\varpi$ is bi-exact, i.e. there is $F\in \Omega^{0,0}(M\mid M)$ such that $\varpi=\mathrm d^L\mathrm d^RF$, then $(M,g^\varpi,\nabla^\varpi)$ is a statistical manifold. Conversely, if the contrast bi-form $\varpi$ induces a statistical manifold structure on $M$, then there is $F\in \Omega^{0,0}(M\mid M)$ such that $\varpi$ and $\mathrm d^L\mathrm d^RF$ produce the same pseudo-Riemannian metric and affine connection. Moreover, $F$ can be chosen in such a way that $\iota^\ast \mathrm d^LF=0, \iota^\ast \mathrm d^RF=0$ and $\iota^\ast F=0$.
\end{theorem}
\begin{remark}
    Interpreting $S$ as a fiber-wise linear function $M \times TM \to \mathbb{R}$ gives back Henmi and Matsuzoe's definition of a pre-contrast function   \cite{ref_HM}. Likewise, the function $F$ recovers Eguchi's idea of contrast function \cite{ref_Egu}.
\end{remark}
\section{An example: teleparallel manifolds}
This section presents an easy yet significant class of Lauritzen manifolds, allowing us to explicitly approach the inverse problem in information geometry within the formalism of bi-forms.
\subsection{Teleparallel manifolds}
Consider a finite-dimensional parallelizable pseudo-Riemannian manifold $(M,g)$, and fix a global frame $\mathcal{B} = \{\alpha^j\}_{j=1}^{\dim M}$ of $T^\vee M$. It is well known there exists  a unique affine connection $\nabla$ on $M$ such that each $\alpha^j$ is $\nabla$-covariantly constant, and we say that $\nabla$ is teleparallel with respect to $\mathcal{B}$. The torsion of the connection  $\nabla$ is directly related to the closedness of elements in $\mathcal{B}$: the connection is torsion-free if and only if each $\alpha^j$ is closed. The $g$-dual connection $\nabla^\dag$ is characterized by the property that each $g$-gradient vector fields of $\mathcal{B}$ is $\nabla^\dag$-parallel. In general, both $\nabla$ and $\nabla^\dag$ may possess torsion, and such structures provide a natural class of Lauritzen manifolds \cite{ref_CdC,ref_ZK19}.
 
 We now study the inverse problem in information geometry for $(M,g,\nabla)$, i.e., we seek a contrast bi-form $\varpi$ such that:
\begin{equation} \label{Eq:integration}
\big(M,g^\varpi,\nabla^\varpi\big) = (M,g,\nabla)\,.
\end{equation}
Although  this problem makes sense for arbitrary Lauritzen manifolds, we will construct explicitly a solution in the teleparallel setting. To this end, we exploit the presence of a global frame on $TM$, namely the $g$-gradient vector field of elements of $\mathcal{B}$. For each $\alpha^j$, we denote by $Z_j$ the $g$-gradient of $\alpha^j$, i.e. the unique vector field on $M$ satisfying:
\begin{equation}
\mathbf i_{Z_j}g=\alpha^j\,.
\end{equation}
This allows us to express any $(1,1)$-bi-form on $M$ as:
\begin{equation}
    \varpi = \sum_{i,j=1}^{\dim M}\varpi_{ij} \, \pi_L^\ast(\beta^i) \otimes \pi_R^\ast(\beta^j)\,,
\end{equation}
where $\{\beta^i\}_{i=1}^{\dim M}$ is the dual co-frame of $\{Z_j\}_{j=1}^{\dim M}$, and $\varpi_{ij}\in \Omega_{\Delta_M}^{0,0}(M\mid M)$. In such a case, \eqref{Eq:integration} reads as the system of equations:
\begin{equation}
    \begin{cases}
        \iota^\ast\big(\mathcal L_{Z^L}\varpi_{ij}\big) = g\big(\nabla_Z Z_i, Z_j\big)\\
        \iota^\ast \varpi_{ij} = g_{ij}
    \end{cases} \qquad i,j\in \{1,\dots,\dim M\}
\end{equation}
where $Z$ is an arbitrary vector field on $M$ and $g_{ij}=g(Z_i,Z_j)$. Alternatively, one can show that the problem is equivalent to:
\begin{equation}
    \begin{cases}
        \iota^\ast\mathrm d^L\varpi_{ij}=\mathrm dg_{ij}\\
        \iota^\ast \varpi_{ij} = g_{ij}
    \end{cases} \qquad i,j\in \{1,\dots,\dim M\}\,,
\end{equation}
A straightforward solution is given by $\varpi_{ij} = \pi_L^\ast g_{ij}$, yielding:
\begin{equation}\label{Eq: contrast bi-form}
    \varpi = \sum_{j=1}^{\dim M}\pi_L^\ast(\alpha^j) \otimes \pi_R^\ast(\beta^j)\,,
\end{equation}
where we used the relation:
\begin{equation}
    \alpha^j=\sum_{i=1}^{\dim M}g_{ij}\,\beta^i\,.
\end{equation}
\subsection{The manifold of faithful quantum states}
An important example of teleparallel manifold is $\mathcal{S}(\mathcal{H})$, the manifold of faithful quantum states of a $d$-level quantum system described by a $d$-dimensional complex Hilbert space $\mathcal H$. A Riemannian metric is given by a quantum monotone metric tensor $g^f$ associated to an operator monotone function $f$ \cite{C-DC-DN-V-2022,ref_Pez} according to:
\begin{equation}
g^f(X,Y)(\rho) = \operatorname{Tr} \big( X_\rho \, K_\rho^f(Y_\rho) \big)\,,
\end{equation}
where $\rho\in\mathcal{S}(\mathcal{H})$, $X$ and $Y$ are vector fields on $\mathcal{S}(\mathcal{H})$, $X_{\rho}$ and $Y_{\rho}$ are identified with traceless Hermitian operators on $\mathcal{H}$, and $K^{f}_{\rho}$ is the superoperator $K^{f}_{\rho}\colon\mathcal{B}(\mathcal{H})\rightarrow\mathcal{B}(\mathcal{H})$ given by:
\begin{equation}
    K^{f}_{\rho}(a)=\left(f(L_{\rho}R_{\rho}^{-1})R_{\rho}\right)^{-1}(a)\,,
\end{equation}
with $L_{\rho}(a)=\rho a$, $R_{\rho}(a)=a\rho$, and $f$ an operator monotone function satisfying $tf(t^{-1})=f(t)$.

Finally, $\nabla$ is taken to be the connection associated with the natural convex structure of $\mathcal{S}(\mathcal{H})$ 
. Let $\{A_i\}_{i=0}^{d^2-1}$ be a basis of the real vector space $\mathcal B_{\operatorname{sa}}(\mathcal H)$ of Hermitian operators such that $A_0=\mathbb I$, $\operatorname{Tr}(A_i)=0$ and $\operatorname{Tr}(A_iA_j)=\delta_{ij}$ for each $i,j\in \{1,\dots,d^2-1\}$. 
The connection $\nabla$ is teleparallel with respect to $\{\mathrm d e_{A_i}\}_{i=1}^{d^2-1}$, where:
\begin{equation}
 e_{A_i}(\rho) = \operatorname{Tr}(A_i \rho)\,.
\end{equation}
Although $\nabla$ is torsion-free, the torsion of $\nabla^\dag$ depends on the choice of the quantum monotone function $f$; the torsion vanishes if $f$ corresponds to the Bogoliubov-Kubo-Mori metric.

The solution \eqref{Eq: contrast bi-form} in this example reads:
\begin{equation}
    \varpi=\sum_{i=1}^{d^2-1}\pi_L^\ast (\mathrm d  e_{A_i})\otimes \pi_R^\ast( \beta^i)\,,
\end{equation}
where $\{ \beta^i\}_{i=1}^{d^2-1}$ denotes the dual co-frame of $  g^f$-gradient vector fields of the family $\{  e_{A_i}\}_{i=1}^{d^2-1}$ \cite{ref_CdC}.

Since $\nabla$ is torsion-free, the potential generating $(\mathcal{S}(\mathcal{H}),g^f,\nabla)$ can be found in terms of a $(0,1)$-bi-form $S$:
\begin{equation}
    S=\sum_{i=1}^{d^2-1}(\pi_L^\ast-\pi_R^\ast)(e_{A_i})\,\pi_R^\ast(\beta^i)\,.
\end{equation}
If $f$ is the Bogoliubov-Kubo-Mori quantum monotone function, then $(M,g^f,\nabla)$ is a statistical manifold, and so the potential can be found in terms of a contrast function on $M$, the \textit{von Neumann-Umegaki relative entropy}, denoted by $\operatorname{vNU}$:
\begin{equation}
    \operatorname{vNU}(\rho,\sigma)=\operatorname{Tr}(\rho\,\log \rho-\rho\,\log \sigma)\,.
\end{equation}

This example illustrates how the formalism of bi-forms provides a powerful tool for analysing geometric structures in quantum information geometry.

\begin{credits}
\subsubsection{\ackname} This work was supported by the FRA 2022 project GALAQ (Geometric and Algebraic Aspects of Quantization) of the University of Naples Federico II. Further support was provided by the Madrid Government under the Multiannual Agreement with UC3M in the framework of “Research Funds for Beatriz Galindo Fellowships” (C\&QIG-BG-CM-UC3M), within the V PRICIT programme (Regional Programme of Research and Technological Innovation), and through the project TEC-2024/COM-84 QUITEMAD-CM. Support from the COST Action CaLISTA CA21109 (European Cooperation in Science and Technology), the PRIN 2022 project 2022XZSAFN (CUP: E53D23005970006) funded by the Next Generation EU programme, and GNSAGA (INdAM) is also acknowledged.

\end{credits}
%
%
%

\begin{thebibliography}{8}

\bibitem{ref_AA}
Ay, N., Amari, S.: A novel approach to canonical divergences within information geometry. Entropy \textbf{17}(2), 811--842 (2015).

\bibitem{ref_CdC}
Ciaglia, F.M., Di Cosmo, F., Ibort, A., Marmo, G.: G-dual teleparallel connections in information geometry. Information Geometry \textbf{7}(4), 587--608 (2023).

\bibitem{C-DC-DN-V-2022}
Ciaglia, F.M., Di Cosmo, F., Di Nocera, F., Vitale, P.: Monotone metric tensors in quantum information geometry. International Journal of Geometric Methods in Modern Physics \textbf{21}(10), 2440004 (2024).

\bibitem{ref_Egu}
Eguchi, S.: Geometry of minimum contrast. Hiroshima Math. J. \textbf{22}(3), 631--647 (1992).

\bibitem{ref_EB_I}
Einstein, A., Bargmann, V.: Bivector fields. Ann. Math. \textbf{45}(1), 1--14 (1944).

\bibitem{ref_EB_II}
Einstein, A.: Bivector fields II. Ann. Math. \textbf{45}(1), 15--23 (1944).

\bibitem{ref_HM}
Henmi, M., Matsuzoe, H.: Statistical manifolds admitting torsion and partially flat spaces. In: Nielsen, F. (ed.) Geometric Structures of Information, Proceedings of the Geometric Science of Information Conference, pp. 37--50. Springer, Cham (2019).

\bibitem{ref_HM_2017}
Henmi, M., Matsuzoe, H.: Statistical manifolds admitting torsion, pre-contrast functions and estimating functions. In: Nielsen, F., Barbaresco, F. (eds.) Information Geometry and Its Applications, pp. 301--315. Springer, Cham (2017).

\bibitem{ref_J}
Jenčová, A.: Geometry of quantum states: dual connections and divergence functions. Reports on Mathematical Physics \textbf{47}(1), 121--138 (2001).

\bibitem{ref_Kurose}
Kurose, T.: Statistical manifolds admitting torsion. Fukuoka Univ., Geometry and Something (2007) (in Japanese).

\bibitem{ref_Lau}
Lauritzen, S.L.: Statistical manifolds. In: Amari, S. (ed.) Differential geometry in statistical inferences, IMS Lecture Notes Monograph Series \textbf{10}, pp. 163--216. Institute of Mathematical Statistics, Hayward (1987).

\bibitem{ref_Matu}
Matumoto, T.: Any statistical manifold has a contrast function—On the $C^3$-functions taking the minimum at the diagonal of the product manifold. Hiroshima Mathematical Journal \textbf{23}(2), 327--332 (1993).

\bibitem{ref_NSS}
Nickerson, H.K., Spencer, D.C., Steenrod, N.E.: Advanced Calculus. Van Nostrand, New York (1959).

\bibitem{ref_Pez}
Petz, D.: Monotone metrics on matrix spaces. Linear Algebra and its Applications \textbf{244}, 81--96 (1996).

\bibitem{ref_Ruse}
Ruse, H.S.: Absolute partial differential calculus. Proc. Lond. Math. Soc. \textbf{2}(1), 194--215 (1932).

\bibitem{ref_ZK}
Zhang, J., Khan, G.: Statistical mirror symmetry. Differential Geometry and its Applications \textbf{73}, 587--608 (2020).

\bibitem{ref_ZK19}
Zhang, J., Khan, G.: From Hessian to Weitzenb{\"o}ck: manifolds with torsion-carrying connections. Information Geometry \textbf{2}(1), 77--89 (2019).

\end{thebibliography}
%

\end{document}